\documentclass{amsart}
\usepackage{amsmath,amssymb,verbatim,color,enumerate,graphicx,caption,verbatim,textcomp}

\newtheorem{thm}{Theorem}[section]
\newtheorem{cor}[thm]{Corollary}
\newtheorem{lem}[thm]{Lemma}

\newtheorem{thmx}{Theorem}

\theoremstyle{remark}

\newcommand{\ds}{\displaystyle}
\newcommand{\dint}{\ds \int}

\newcommand{\eqskip}{ \vspace*{2mm}\\ }
\newcommand{\fr}[2]{\frac{\ds #1}{\ds #2}}
\newcommand{\so}{{\rm o}}

\title[Extremal eigenvalues of the biharmonic operator on rectangles]
{Extremal eigenvalues of the Dirichlet biharmonic operator on rectangles}
\author{D.\ Buoso and P.\ Freitas}

\address{
\'Ecole Polytechnique F\'ed\'eral de Lausanne, EPFL SB MATH, SCI-SB-JS, MA B3 514 (B\^atiment MA), Station 8, CH-1015 Lausanne, Switzerland}
\email{davide.buoso@epfl.ch}

\address{Departamento de Matem\'atica, Instituto Superior T\'ecnico, Universidade de Lisboa, Av. Rovisco Pais 1,
P-1049-001 Lisboa, Portugal {\rm and}
Grupo de F\'isica Matem\'{a}tica, Faculdade de Ci\^encias, Universidade de Lisboa,
Campo Grande, Edif\'icio C6, P-1749-016 Lisboa, Portugal}
\email{psfreitas@fc.ul.pt}


\begin{document}

\begin{abstract}
We study the behaviour of extremal eigenvalues of the Dirichlet biharmonic operator over rectangles with a given
fixed area. We begin by proving that the principal eigenvalue is minimal for a rectangle for which the ratio
between the longest and the shortest side lengths does not exceed $1.066459$. We then consider the sequence formed
by the minimal $k^{\rm th}$ eigenvalue and show that the corresponding sequence of minimising rectangles
converges to the square as $k$ goes to infinity.
\end{abstract}

\maketitle

\section{Introduction}

The asymptotic behaviour of extremal eigenvalues of the Laplace operator has received some attention in
the mathematical literature in recent years, starting with the proof that in the case of rectangles with fixed area
and Dirichlet boundary conditions extremal rectangles converge to the square as the order of the eigenvalue goes to
infinity~\cite{antfre}. This result has been generalised to rectangular parallelepipeds in higher dimensions and
to Neumann boundary conditions~\cite{berbucgit,bergit,gitlar,mars}. All these results are based on the relation
between this eigenvalue problem and lattice point problems, and some generalisations along these lines have also
began to appear~\cite{arilau,lauliu,marste}. 

A natural question is, of course, whether or not such results also extend to more general domains. That the problem in
the most general case of bounded domains is expected to be difficult is a consequence of the result by Colbois and El Soufi
which relates this to a statement equivalent to P\'{o}lya's conjecture~\cite{colsou}. There are, however,
some results under weaker conditions. By imposing a surface area restriction instead of a volume restriction, it is
possible to show that in the planar case there is convergence of extremal domains to the disk~\cite{bucfre} and, by
considering averages instead of single eigenvalues, it then becomes possible to show convergence of such averages
or even, in some cases, of the corresponding extremal domains~\cite{fre,lars}.

From a physical perspective, this type of problem may be seen as that of finding the shape for which the number of
modes allowed below a given frequency is extremal. Then, the existing results and corresponding proofs indicate
that in the high-frequency regime this behaviour is again determined by the classical geometric isoperimetric inequality,
just as in the case of the extremal domain for the first eigenvalue. In some sense this is not unexpected,
as the first two terms in the Weyl law depend only on the volume and surface measures. On the other hand, it is not
clear that these two terms should play the dominant role in this setting. Furthermore, it was recently shown
that, in the case of Robin boundary conditions with a positive boundary parameter, eigenvalues satisfy (nontrivial) P\'{o}lya-type
inequalities with lower bounds for the $k^{\rm th}$ eigenvalue, in spite of the fact that the first two terms in
the corresponding Weyl law for Robin eigenvalues are the same as those of the Neumann problem -- see~\cite{anfrke,freken}.

In this paper we are interested in studying the above mentioned problem in the case of the biharmonic operator with Dirichlet
boundary conditions. Determining extremal domains for the biharmonic operator with Dirichlet boundary conditions is
a notoriously difficult problem, as may be seen from the fact that proofs of the corresponding Faber-Krahn inequality
exist only in two and three spatial dimensions~\cite{ashben,nad} -- see also \cite{ashlau} for a discussion about the limitations
arising in higher dimensions. Part of this difficulty stems from the fact that
the first eigenfunction is no longer necessarily of one sign.
In the case of rectangles this becomes particularly relevant as it is known that the first eigenfunction does indeed
change sign, including in the case of the square which is the natural candidate for the minimiser of the first eigenvalue
(see e.g., \cite{dav, coff,kkm}). Indeed, it is not known if the square minimises
the first Dirichlet eigenvalue of the biharmonic operator among all rectangles of a given area. Our first result is that there exists
one global minimiser for this eigenvalue
and that it must be quite close to the square. More precisely, we have the following
\begin{thmx}
 There exists a global minimiser for the first Dirichlet eigenvalue of the biharmonic operator over rectangles of
 fixed area. Furthermore, the quotient between the lengths of the largest and the smallest sides of the extremal rectangle
 does not exceed $1.066459$. 
\end{thmx}

At the high-frequency end of the spectrum, and since the proofs now do not rely on any such properties, we are
able to show that there is convergence to the square.
\begin{thmx}
 Let $q_{k}^{*}$ denote the quotient between the lengths of the largest and the smallest sides of a rectangle minimising
 the $k^{\rm th}$ eigenvalue of the Drichlet biharmonic operator over rectangles of fixed area. Then
 \[
  \lim_{k\to\infty} q_{k}^{*} = 1.
 \]
\end{thmx}

Some other results in the spirit of those for the Dirichlet Laplacian mentioned above are more or less straighfrorward consequences of
the corresponding original result. They include the case of fixed perimeter and the Colbois-El Soufi results on the sequence of minimisers.
For completeness, we collect these in Section~\ref{collection}.

\section{Background and notation}

Let $\Omega$ be a (smooth bounded) domain in $\mathbb R^N$, $N\ge2$. The Dirichlet eigenvalue problem for the biharmonic operator (clamped plate problem)
is given by
\begin{equation}
\label{dirichlet}
\left\{\begin{array}{ll}
\Delta^2u=\lambda u, & {\rm \ in\ }\Omega,\\
u=\fr{\partial u}{\partial \nu}=0, & {\rm \ on\ }\partial \Omega,
\end{array}\right.
\end{equation}
with the corresponding weak formulation being
\begin{equation*}
\int_{\Omega}\Delta u\Delta\phi=\lambda\int_\Omega u\phi,\ \forall \phi\in H^2_0(\Omega).
\end{equation*}
The eigenvalues of the above problem may be characterised by a variational principle of the form
\begin{equation*}
\lambda_k(\Omega)=\min_{0\neq u\in V\subset H^2_0(\Omega)}\max_{\dim V=k}\fr{\int_\Omega (\Delta u)^2}{\int_\Omega u^2}
\end{equation*}
and it is known that the sequence $\lambda_{k}$ satisfies
\[
 0<\lambda_{1} \leq \lambda_{2} \leq \cdots \to +\infty \mbox{ as } k\to \infty.
\]
Under certain geometric conditions on a piecewise smooth domain $\Omega$, which are satisfied by rectangles (cf.\ \cite{vass}), the corresponding Weyl asymptotics for
problem~\eqref{dirichlet} on planar domains may be seen from~\cite[formulas (6.2.1) and (6.2.2)]{safvass} to be 
\begin{equation}
\label{weyl2}
\lambda_k=\fr{16\pi^2}{|\Omega|^2}k^2+ \fr{16\pi^{\frac 3 2}|\partial \Omega|}{|\Omega|^{\frac 5 2}}\left(1+\frac{\Gamma\left(\frac34\right)}{\sqrt{\pi}\Gamma\left(\frac54\right)}\right)k^{\frac 3 2} +\so\left(k^{\frac 3 2}\right),
\end{equation}
or equivalently
\begin{equation}
\label{weyl22}
\lambda_k^{\frac 1 2}=\fr{4\pi}{|\Omega|}k+ \fr{2\pi^{\frac 1 2}|\partial \Omega|}{|\Omega|^{\frac 3 2}}\left(1+\frac{\Gamma\left(\frac34\right)}{\sqrt{\pi}\Gamma\left(\frac54\right)}\right)k^{\frac 1 2} +\so\left(k^{\frac 1 2}\right),
\end{equation}
where $|\Omega|$ and $|\partial\Omega|$ denote the $2$-dimensional measure of $\Omega$ and
the $1$-dimensional measure of its boundary $\partial\Omega$, respectively.

As in the case of the Dirichlet Laplacian, it is also possible to obtain lower bounds of Li-Yau type and we have that the following holds for general smooth domains
(see \cite[formula (1.9)]{kukutan})
\begin{equation}
\label{kukutang}
\lambda_k\ge \fr {16 N\pi^4} {N+4}\left(\fr k{\omega_N|\Omega|}\right)^{\frac 4 N}.
\end{equation}

\section{The first eigenvalue: the square is (almost) the minimising rectangle}

In this section, we focus our attention on the question of determining the minimal possible value for the first eigenvalue of problem~\eqref{dirichlet} among 
rectangles with a given fixed area. Without loss of generality, we fix the area to be one so that our class of admissible rectangles may be written as
$$
\mathcal{R}=\{\text{Rectangles with side lengths $a$ and $1/a$, for $a\in[1,+\infty)$}\}.
$$

We recall that the biharmonic operator appearing in problem~\eqref{dirichlet} is invariant under rotations and translations as is
the case for the Laplace operator, and hence the spectrum is the same for any rectangle of given edges.  Thus, and due to symmetry
considerations, we expect the square to be an extremal point for $\lambda_{1}$, for otherwise there would have to be an infinite number of oscillations
for $a$ close to one. However, most other fundamental properties of the Dirichlet
Laplacian are not shared by the biharmonic operator. For instance, we know that the first biharmonic eigenvalue is not necessarily
simple in general and, although the first eigenvalue of the square is expected to be simple, the only results in this direction are of a numerical
nature~\cite{wiener}.  Furthermore, some useful properties such as separation of variables are not available for rectangles, and therefore we cannot 
characterize either its eigenvalues or eigenfnctions explicitly in terms of known functions. This, together with the lack of positivity for the first eigenfunction
already mentioned in the Introduction, transforms what is a trivial problem in the case of the Laplacian into a quite hard problem. 

Our approach will make use of the sharp estimates provided by Owen~\cite{owen} in order to narrow down the search of the minimiser
to a small neighbourhood of the square $(a=1)$.

Let us denote by $R_a$ the rectangle $R_a=[0,a]\times[0,1/a]$, and write $\lambda_1(a)=\lambda_1(R_a).$
We recall the following estimate from~\cite[Table 4]{wiener}.
\begin{lem}
The first eigenvalue of problem \eqref{dirichlet} satisfies
$$
1294.933940\le\lambda_1(1)\le \Lambda:=1294.933988.
$$
\end{lem}
We also recall a lower bound from~\cite[Theorem 2]{owen}. 
\begin{lem}
For any $a\ge 1$ we have
\begin{equation}
\label{La}
\lambda_1(a)\ge L(a)=\rho(\pi^2 a^4)a^{-4}+\rho(\pi^2 a^{-4})a^4-2 \pi^4,
\end{equation}
where $\rho(\alpha)$ is the first eigenvalue of the following problem
\begin{equation}
\label{prob1d}
\left\{\begin{array}{ll}
y''''-2\alpha y''=\lambda y, & \text{in\ }(0,1),\\
y(0)=y(1)=y'(0)=y'(1)=0,
\end{array}\right.
\end{equation}
and $\rho(\alpha)$ is an increasing function for positive $\alpha$.
\end{lem}

Finally, we will also need the following
\begin{lem}
The function $L(a)$ defined in \eqref{La} is strictly increasing in $a$ for $a>1$.
\end{lem}
\begin{proof}
In order to prove that $L'(a)>0$ for $a>1$, we denote by $v_t$ an eigenfunction associated with $\rho(\pi^2 t)$ in~\eqref{prob1d}
such that $\|v_t\|_{L^2}=1$. Then
\[
 \rho(\pi^2 t) = \int_{0}^{1} (v_{t}'')^2 + 2\pi^2 t \int_{0}^{1} (v_{t}')^2
\]
and
$$
\rho'(\pi^2 t)=2\int_0^1(v_t')^2.
$$
Writing
$$
F(t)=\rho(\pi^2 t)t^{-1}+\rho(\pi^2 t^{-1})t,
$$
we have
$$
\begin{array}{lll}
F'(t) & = & \fr{2\pi^{2}}{t} \dint_{0}^{1}\left(v_t'\right)^2 - \fr{1}{t^2}\left[\dint_{0}^{1} (v_{t}'')^2 + 2\pi^2 t \int_{0}^{1} (v_{t}')^2\right]\eqskip
& & \hspace*{5mm}-\fr{\pi^2}{t}\dint_{0}^{1}(v_{t^{-1}}')^2 + \left[ \int_{0}^{1} (v_{t^{-1}}'')^2 + \fr{2\pi^2}{t} \int_{0}^{1} (v_{t^{-1}}')^2\right] \eqskip
& = & -\fr1{t^2}\dint_0^1(v_t'')^2+\dint_0^1(v_{t^{-1}}'')^2 \eqskip
& = & \fr{1}{t^2}\left[t^2\dint_0^1(v_{t^{-1}}'')^2+2\pi^2t\dint_0^1(v_{t}')^2-\rho(\pi^2t)\right].
\end{array}
$$
At this point we observe that
$$
\int_0^1(v_{t^{-1}}'')^2\ge\min_{\substack{v\in H^2_0 \\ \|v\|_{L^2}=1}}\int_0^1(v'')^2=\rho(0)=\int_0^1(v_{0}'')^2.
$$
Moreover, since
$$
\rho(\pi^2 t)=\min_{\substack{v\in H^2_0 \\ \|v\|_{L^2}=1}}\left[\dint_0^1(v'')^2+2\pi^2t\int_0^1(v')^2\right],
$$
then, also using the Poincar\'e inequality
$$
\int_0^1(v')^2\ge\pi^2\int_0^1v^2\ \ \forall v\in H^1_0(0,1),
$$
we get
\[
\begin{array}{lll}
t^2\dint_0^1(v_{t^{-1}}'')^2+2\pi^2t\dint_0^1(v_{t}')^2-\rho(\pi^2t) & 
\ge &  (t^2-1)\dint_0^1(v_{0}'')^2-2\pi^2t\dint_0^1(v_{0}')^2+2\pi^4t\eqskip
& \ge &  \pi^2(t^2-2t-1)\dint_0^1(v_{0}')^2+2\pi^4t\eqskip
& \ge & \pi^4(t^2-1).
\end{array}
\]
Hence $F'(t)>0$ for $t>1$, and the result now follows by observing that $L(a)=F(a^4)-2\pi^4$.
\end{proof}

Our strategy is to find bounds for $a\in[1,+\infty)$ such that
\begin{equation}
\label{firstattempt}
\Lambda<L(a).
\end{equation}
In fact, if $\hat{a}$ is a solution of $\Lambda=L(a)$, then we obtain that the solution of
\begin{equation}
\label{miniproblem}
\min_{a\ge 1} \lambda_1(a)
\end{equation}
has to be a rectangle $R_a$ with $a\in[1,\hat{a})$, the precision of this bound being strictly related to the
precision of the bounds $\Lambda$ and $L(a)$. We remark that there
exists at least one solution to problem~\eqref{miniproblem}, since $\lim_{a\to\infty}\lambda_1(a)=\infty$.

In order to find the smallest value $\hat a$ satisfying \eqref{firstattempt}, we implement a bisection procedure in the software
Mathematica\texttrademark\ starting from $L(2)=9442.68$.

\begin{thm}
Problem \eqref{miniproblem} admits (at least) one minimiser $a^*$ in the interval
$$
a^*\in [1,\hat a),
$$
where $\hat a$ is the solution of equation $\Lambda=L(a)$, lying in the interval
$$
\hat a \in [1.03269,1.032695).
$$
\end{thm}

We observe that this method allows us to say that the minimiser has to be very close to the square, but we cannot go any further below the threshold
$\hat a$. Having some additional information such as convexity of the first eigenvalue with respect to this perturbation, or simplicity of
eigenvalues for rectangles close to the square, would allow for a more complete result. Some numerical simulations based on the method
of fundamental solutions~\cite{alan}, however, give support to the conjecture that the square is the actual global minimizer among all rectangles
of unit area (see Fig.~\ref{figura}).
We also note that, even though the general form of the 
shape derivative for eigenvalues of problem \eqref{dirichlet} is known (see e.g., \cite{buo,buolam}), its value is extremely difficult to estimate for 
the square and for rectangles in general, since, in contrast with the Dirichlet Laplacian and as mentioned above, the explicit
form of the eigenfunctions is not known.

\begin{figure}[ht]
\centering
\includegraphics[width=1.\textwidth]{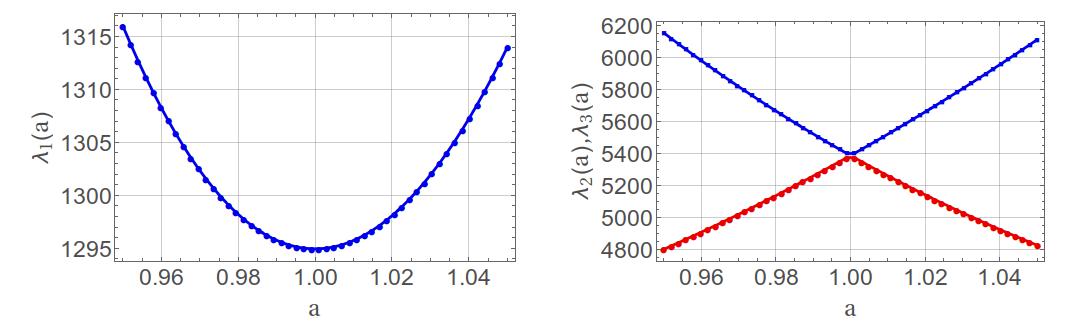}
\caption{On the left, $\lambda_1(a)$ for rectangles with different sides around a square ($a=1)$; on the right, the same for $\lambda_2,\lambda_3$.}
\label{figura}
\end{figure}

Since the bounds obtained by the above methods are not explicit and still require the solution of a transcendental equation in each case,
we conclude this section with a simple bound which, although not as accurate, has the advantage that it only requires the determination of one
such root.

\begin{thm}
 The first eigenvalue of problem~\eqref{dirichlet} on a rectangle $R_{a}$ satisfies the bound
 \[
  \lambda_{1}(a) \geq \omega_{1}^4 \left( a^4 + a^{-4}\right) + 2\pi^4,
 \]
where $\omega_{1}\approx 4.73004$ is the first positive root of the equation $\cos(\omega)\cosh(\omega) = 1.$
\end{thm}
\begin{proof}
 On a rectangle $R_{a}$ we have
 \[
 \begin{array}{lll}
  \dint_{R_{a}} \left( \Delta u\right)^2 & = & \dint_{R_{a}} u_{xx}^2 + u_{yy}^2 + 2 u_{xx}u_{yy} \:\:{\rm d}x{\rm d}y\eqskip
  & = & \dint_{R_{a}} u_{xx}^2 + u_{yy}^2 + 2 u_{xy}^2 \:\:{\rm d}x{\rm d}y.
 \end{array}
\]
We thus have
\[
 \begin{array}{lll}
  \lambda_{1}(a) & = &  {\ds \inf_{0\neq u\in H_{0}^{2}(R_{a})}} \fr{ \dint_{R_{a}} \left( \Delta u\right)^2 }{\dint_{R_{a}} u^2}\eqskip
  & \geq &  {\ds \inf_{0\neq u\in H_{0}^{2}(R_{a})}}\fr{ \dint_{R_{a}} u_{xx}^2 + u_{yy}^2 }{\dint_{R_{a}} u^2}  +
  2  {\ds \inf_{0\neq u\in H_{0}^{2}(R_{a})}} \fr{ \dint_{R_{a}} u_{xx}u_{yy}}{\dint_{R_{a}} u^2}.
 \end{array}
\]
The first term corresponds to the first eigenvalue of the problem
\[
 u_{xxxx}+u_{yyyy} =\gamma u
\]
in $R_{a}$, together with Dirichlet boundary condtions. Note that the operator $\partial_{xxxx}+\partial_{yyyy}$ is strongly
elliptic, and therefore admits a purely discrete specrtum accumulating at infinity, with the usual minimax characterization for
its eigenvalues (see also \cite[Theorem 9, page 176]{burenkov}). For this problem it is possible to separate variables in the
usual way to obtain
\[
 \gamma_{1}(a) = \omega_{1}^4 (a^4 + a^{-4}).
\]
Regarding the second term, we have
\[
 \begin{array}{lll}
 \dint_{R_{a}} u_{xy}^2 \:\:{\rm d}x{\rm d}y & = & \dint_{0}^{a^{-1}}\dint_{0}^{a} \left[ \left( u_{y}\right)_{x}\right]^2 \:\:{\rm d}x{\rm d}y\eqskip
  & \geq & \dint_{0}^{a^{-1}}\fr{\pi^2}{a^2} \dint_{0}^{a} u_{y}^2 \:\:{\rm d}x{\rm d}y\eqskip
  & \geq & \pi^4 \dint_{0}^{a^{-1}} \dint_{0}^{a} u^2 \:\:{\rm d}x{\rm d}y\eqskip
  & = & \pi^4 \dint_{R_{a}} u^2 \:\:{\rm d}x{\rm d}y,
 \end{array}
\]
and putting these two estimates together yields the result.
\end{proof}

\section{High-frequency limit: minimising rectangles converge to the square}

In this section we are interested in what happens to the minimiser of $\lambda_k$ as $k\to\infty$. To simplify notation, in what follows we write $\lambda_k(a)$
for $\lambda_k(R_a)$.  As in~\cite{antfre}, we start with a lower bound for $\lambda_k(a)$ which, for the case of rectangles, improves upon what would be
obtained by a direct application of P\'{o}lya's bound for tiling domains.

\begin{lem}
\label{lemmacit}
For any $a\ge 1$ we have
\begin{equation}
\label{ineq}
\lambda_k^{\frac 1 2}(a)\ge 4\pi k +2 a \lambda_k^{\frac 1 4}(a)-\frac{4 \sqrt{2\pi}}{3\sqrt{3}}a^{\frac 3 2}\lambda_k^{\frac 1 8}(a).
\end{equation}
\end{lem}

\begin{proof}
From \cite[Theorem 3.1]{antfre} we have
\begin{equation}
\label{original}
\lambda_k^D(a)\ge 4\pi k +2 a \left[\lambda_k^D(a)\right]^{\frac 1 2}-\frac{4 \sqrt{2\pi}}{3\sqrt{3}}a^{\frac 3 2}\left[\lambda_k^D(a)\right]^{\frac 1 4},
\end{equation}
where $\lambda_k^D(a)$ is the $k$-th eigenvalue of the Dirichlet problem
\begin{equation*}
\left\{\begin{array}{ll}
-\Delta u=\lambda u, & {\rm \ in\ } R_a,\\
u=0, & {\rm \ on\ }\partial R_a.
\end{array}\right.
\end{equation*}

At this point we observe that
\begin{equation}\label{dirnav}
\lambda_k(a)\ge\left[\lambda_k^D(a)\right]^2,
\end{equation}
since $\left[\lambda_k^D(a)\right]^2$ is the $k$-th eigenvalue of the Navier problem
\begin{equation}
\label{navierpb}
\left\{\begin{array}{ll}
\Delta^2u=\lambda u, & {\rm \ in\ }R_a,\\
u=\Delta u=0, & {\rm \ on\ }\partial R_a,
\end{array}\right.
\end{equation}
and may be characterized as
\begin{equation*}
\left[\lambda_k^D(a)\right]^2=\min_{0\neq u\in V\subset H^2(R_a)\cap H^1_0(R_a)}\max_{\dim V=k}\frac{\int_{R_a} (\Delta u)^2}{\int_{R_a} u^2}.
\end{equation*}
See also \cite[Chapter 2.7]{ggs} for a discussion about the coercivity of problem \eqref{navierpb}.

We rewrite~\eqref{original} as
\begin{equation*}
\lambda_k^D(a)-2 a \left[\lambda_k^D(a)\right]^{\frac 1 2}\ge 4\pi k -\frac{4 \sqrt{2\pi}}{3\sqrt{3}}a^{\frac 3 2}\left[\lambda_k^D(a)\right]^{\frac 1 4}.
\end{equation*}
Using the fact that $t-2a\sqrt{t}$ is increasing in $t$ for $t\ge a^2$ and that $\lambda_k(a)^{\frac 1 2}\ge \lambda_k^D(a)\ge a^2$, we obtain inequality~\eqref{ineq}.
\end{proof}

Let us now set
\begin{equation*}
\lambda_k^*=\min_{a\ge 1}\lambda_k(a).
\end{equation*}
It is clear that the minimum is achieved since $\lambda_k(a)\to \infty$ as $a\to\infty$. We also set $a_k^*$ in such a way that
$$
\lambda_k(a_k^*)=\lambda_k^*.
$$
We remark that, in line with what happens for the Dirichlet Laplacian~\cite{antfre}, $a_k^*$ does not have to be uniquely
defined as a function of $k$. although it would probably be extremely difficult to prove so in this case; however, we can always
choose one particular value for each $k\in\mathbb N$.

Then we have
\begin{thm}
The sequence of optimal rectangular shapes for $\lambda_k$ converges to the square as $k\to\infty$, i.e.,
\begin{equation}
\label{limit}
\lim_{k\to\infty}a_k^*=1.
\end{equation}
\end{thm}
As we just noticed, uniqueness of the optimizer may fail for some $k$; nevertheless, for any possible choice, the limit \eqref{limit} holds.

\begin{proof}
First of all, following the argument used in \cite[Theorem 3.5]{antfre} coupled with the Weyl asymptotic expansion \eqref{weyl22} and Lemma \ref{lemmacit}, we get that
\begin{equation*}
\limsup_{k\to\infty}a_k^*\leq \frac{6\sqrt{3}}{3\sqrt{3}-2\sqrt{2}}\left[1+\frac{\Gamma\left(\frac34\right)}{\sqrt{\pi}\Gamma\left(\frac54\right)}\right],
\end{equation*}
meaning that the sequence $\{a_k^*\}_k$ is bounded.
At this point, using the boundedness of $\{a_k^*\}_k$ and the fact that
$$
\lambda_k^{\frac 1 2}(a)\ge \lambda_k^D(a)\ge\pi^2\left(a^2+\frac{1}{a^2}\right),
$$
and that $t\to t-2\left(a+\fr 1 a\right)\sqrt{t}$ is increasing for $t\ge\pi^2\left(a^2+\fr{1}{a^2}\right)$, from \cite[formula (3.6)]{antfre} we deduce that
\begin{equation*}
\lambda_k^{\frac 1 2}(a)\ge 4\pi k +2 \left(a+\frac 1 a \right) \lambda_k^{\frac 1 4}(a)-C\lambda_k^{\frac \theta 4}(a)-3\pi,
\end{equation*}
for some $\theta\in(0,1)$. 
We thus deduce an inequality of the type of~\cite[inequality (3.7)]{antfre} and, following the same argument as in~\cite[p.~8]{antfre}, we obtain the
deired result.
\end{proof}

\section{Further results\label{collection}}

For completeness, we now collect some results whose proofs are similar to their Laplacian counterparts.


\subsection{Perimeter constraint}
The first of these corresponds to the minimisation of the $k^{\rm th}$ eigenvalue under a perimeter restriction. More precisely, let
$$
\lambda_k^*=\min\{ \lambda_k(\Omega): \Omega\in\mathcal{R}, |\partial\Omega|=\alpha\},
$$
for some fixed value $\alpha>0$, where $\mathcal{R}$ is a family of bounded domains in $\mathbb R^2$. Let also $\Omega_k^*\in\mathcal{R}$ be a
minimiser for $\lambda_k$, i.e.,
$$\lambda_k^*=\lambda_k(\Omega_k^*).$$
We have the following
\begin{thm}\label{general} Let $\alpha>0$ be fixed.
\begin{itemize}
\item[i)] Let $\mathcal{D}$ be the class of open domains in $\mathbb{R}^2$. Then the sequence of optimal domains $\Omega_k^*$ converges to the disk with perimeter $\alpha$.

\item[ii)] Let $\mathcal{P}_{n}$ be the class of polygons having exactly $n$ sides in $\mathbb{R}^2$. Then the sequence of optimal domains $\Omega_k^*$ converges to the regular $n$-gon with perimeter $\alpha$.

\item[iii)] Let $\mathcal{T}$ be the class of tiling domains in $\mathbb{R}^2$. Then the sequence of optimal domains $\Omega_k^*$ converges to the regular hexagon with perimeter $\alpha$.
\end{itemize}
\end{thm}

The proof of Theorem \ref{general} goes along the same lines of the correponding results in~\cite{bucfre}, now using
the first term in the Weyl asymptotics~\eqref{weyl2}, and inequalities~\eqref{kukutang} and~\eqref{dirnav}. We also note that this result can
be extended to a general polyharmonic problem of the form
\begin{equation}\label{polyharmonic}
\left\{\begin{array}{ll}
(-\Delta)^mu=\lambda u, & {\rm \ in\ }\Omega,\\
u=\frac{\partial u}{\partial \nu}=\dots=\frac{\partial^{m-1} u}{\partial \nu^{m-1}}=0, & {\rm \ on\ }\partial \Omega,
\end{array}\right.
\end{equation}
for $m\ge 1$, as formulas \eqref{weyl2} and \eqref{kukutang} can be generalized to this case as well.


\subsection{Subadditivity}
Let us now set
$$
\lambda_k^*=\min\{\lambda_k(\Omega):\Omega\in\mathbb R^N, |\Omega|=1\},
$$
for any $k$.
\begin{thm}\label{subadd}
Let $i_1\le i_1\le \dots\le i_p$ be positive integers such that $i_1+\dots+i_p=k$. Then
$$
(\lambda_k^*)^{\frac N 4}\le (\lambda_{i_1}^*)^{\frac N 4}+\dots+(\lambda_{i_p}^*)^{\frac N 4}.
$$
In particular,
$$
(\lambda_{k+1}^*)^{\frac N 4}-(\lambda_{k}^*)^{\frac N 4}\le (\lambda_{1}^*)^{\frac N 4}.
$$
\end{thm}

The proof of this result can be obtained following that of \cite[Theorem 2.1]{colsou}. We also have the following corollary thanks to Fekete's Lemma (cf.\ \cite{kukutan} for a general statement of the Generalized Polya conjecture).

\begin{cor}\label{gpc}
The following are equivalent.
\begin{itemize}
\item[i)] {\rm (Generalized Polya conjecture)} For any $k$ and for any domain in $\mathbb R^N$,
$$
\lambda_k\ge 16\pi^4\left(\fr{k}{\omega_N|\Omega|}\right)^{4/N}.
$$
\item[ii)] ${\ds \lim_{k\to\infty}}\fr{\lambda_k^*}{k^{4/N}}=\fr{16\pi^4}{\omega_N^{4/N}}$.
\end{itemize}
\end{cor}
We again observe that both Theorem~\ref{subadd} and Corollary~\ref{gpc} may also be stated for the polyharmonic
problem~\eqref{polyharmonic}.


\section*{Acknowledgements}

The authors would like to thank Pedro Antunes for providing the numerical results used to obtain Fig.\ \ref{figura},
and Richard S.\ Laugesen for numerous comments that resulted in an improved version of the manuscript. This work was partially
supported by the Funda\c c\~{a}o para a Ci\^{e}ncia e a Tecnologia (Portugal) through project {\it Extremal spectral quantities
and related problems} (PTDC/MAT-CAL/4334/2014). Most of the research in this paper was carried out while the first author held
a post-doctoral position at the University of Lisbon within the scope of this project. The first author is a member of the Gruppo
Nazionale per l'Analisi Matematica, la Probabilit\`a e le loro Applicazioni (GNAMPA) of the Istituto Naziona\-le di Alta Matematica (INdAM).


\end{document}